\documentclass[11pt,hyp,]{nyjm}
\usepackage{hyperref}
\hypersetup{nesting=true,debug=true,naturalnames=true}
\usepackage{graphicx,amssymb,upref}
\usepackage{tikz-cd}

\usepackage{enumitem}
\usepackage{tensor}

\let\<\langle
\let\>\rangle

\newcommand{\mailurl}[1]{\email{\href{mailto:#1}{#1}}}
\let\uml\"

\title{Rapid Decay for Principal \'Etale Groupoids}

\author{Alex Weygandt}
\address{Department of Mathematics, Texas A\&M Univeristy, 155 Ireland Street, College Station, TX 77840, USA} 
\mailurl{aweygan@tamu.edu}
\thanks{This work was partially supported by NSF 1952693.} 



\newcommand{\R}{\mathbb{R}}
\newcommand{\C}{\mathbb{C}} 
\newcommand{\N}{\mathbb{N}}
\newcommand{\Z}{\mathbb{Z}}
\newcommand{\T}{\mathbb{T}}

\newcommand{\ve}{\varepsilon}
\renewcommand{\phi}{\varphi}

\newcommand{\F}{\mathbb{F}}
\newcommand{\bbB}{\mathbb{B}}

\newcommand{\G}{\mathcal{G}}
\newcommand{\Go}{\mathcal{G}^{(0)}}
\newcommand{\Gt}{\mathcal{G}^{(2)}}
\newcommand{\CcG}{C_c(\mathcal{G})}
\renewcommand{\H}{\mathcal{H}}
\newcommand{\Ho}{\mathcal H^{(0)}}

\newcommand{\E}{\mathcal{E}}

\newcommand{\ScGE}{\Sigma_c(\mathcal{G},\mathcal{E})}
\newcommand{\CrGE}{C^*_r(\mathcal{G},\mathcal{E})}

\DeclareMathOperator{\propa}{prop}
\DeclareMathOperator{\supp}{supp}
\DeclareMathOperator{\id}{id}

\theoremstyle{plain}
    \newtheorem{theorem}{Theorem}[section]
    \newtheorem{lemma}[theorem]{Lemma}
    \newtheorem{corollary}[theorem]{Corollary}
    \newtheorem{proposition}[theorem]{Proposition}
\theoremstyle{definition}
    \newtheorem{definition}[theorem]{Definition}
    \newtheorem{example}[theorem]{Example}

\begin{document} 
 
\begin{abstract}  
     This work concerns a generalization of property (RD) from discrete groups to twisted \'etale groupoids equipped with a length function.  
     We show that, under the assumption that the \'etale groupoid is principal, twisted property (RD) is equivalent to polynomial growth.  
     This generalizes a result of Chen and Wei concerning  rapid decay for metric spaces with bounded geometry.  
     Additionally, some permanence properties of groupoid (RD) are established.  
\end{abstract} 
\maketitle
\tableofcontents


\section{Introduction}

Let $\Gamma$ be a discrete group, and let $\ell$ be a length function on $\Gamma$.  
This means that $\ell$ is a mapping from $\Gamma$ to the closed half-line $[0,\infty)$ which is subadditive, maps the identity of $\Gamma$ to $0$, and is inverse-invariant.  
We then say that $\Gamma$ has \textit{property (RD)} (short for \textit{rapid decay}) with respect to $\ell$ if there exist constants $C,t\ge0$ such that for all finitely supported $f:\Gamma\to\C$, we have $\|f\|_{C^*_r\Gamma}\le C\|f\|_{\ell,t}$. 
Here, $\|\cdot\|_{C^*_r\Gamma}$ is the operator norm on $\ell^2\Gamma$ given by convolution, and $\|f\|_{\ell,t}=(\sum_{\gamma\in\Gamma}|f(\gamma)|^2(1+\ell(\gamma))^{2t})^{1/2}$ is the weighted $\ell^2$-norm.

First shown to hold for free groups by Haagerup in \cite{Haag1978}, property (RD) for groups was formally defined by Jolissaint in \cite{Joli1990}.  
In that paper, they showed that property (RD) is preserved under taking subgroups and extensions, and is implied by polynomial growth and groups acting geometrically on hyperbolic spaces.  
Building on the latter, de la Harpe showed \cite{delaHa1988} that property (RD) is enjoyed by all Gromov hyperbolic groups.  
Hence property (RD) is satisfied by a large class of groups, and in the past 30+ years many other classes of groups have been shown to satisfy property (RD). 
For a more thorough survey of the history, and of classes of groups for which property (RD) is known, unknown, or conjectured, we refer the reader to \cite{Chat2017}.

One of the first major applications of property (RD) to noncommutative geometry came in \cite{CoMo1990}, where they used it to show that hyperbolic groups satisfy the Novikov conjecture.  
They noted that property (RD) for a finitely generated group $\Gamma$ (with the word length function) gave a means of generating trace-like maps on the $K$-theory of $C^*_r\Gamma$ from well-behaved cocycles on $\Gamma$ (for a groupoid generalization see proposition 6.4 in \cite{Hou2017}).  
The next substantial application of property (RD) came from Lafforgue in \cite{Laff2002},  whose analysis showed the Baum-Connes conjecture holds for a large class of groups satisfying property (RD).  
Here, it was important that property (RD) for $\Gamma$ provides a Banach subalgebra $A$ of $C^*_r\Gamma$ containing $\C\Gamma$ such that (i) the inclusion map $A\hookrightarrow C^*_r\Gamma$ induces an isomorphism on $K$-theory, and (ii) the $A$-norm of elements of $\C\Gamma$ depends only on the magnitude of its coefficients (see \cite{Laff2000}).

Several generalizations and analogues of rapid decay have appeared.  
Considering only functions on the group which are constant on spheres, one obtains what is called radial rapid decay, as considered by Valette in \cite{Val1997}.
In the appendix of \cite{Math2006}, Chatterji defines a rapid decay property for groups with a given $2$-cocycle.  
In another direction, one can consider rapid decay for representations of groups on $L^p$ spaces for $p\ne2$, as done in \cite{liao2017}.

Leaving the realm of groups, one can study a rapid decay property for metric spaces, as done in \cite{Chen2003} and \cite{JiYu2020} (also considered below).  
In another direction, (RD) was also generalized to the setting of quantum groups, as in \cite{Verg2007}. 
In \cite{Hou2017}, property (RD) was extended to the setting of \'etale groupoids.  
In their work, it is shown that several useful consequences of the group rapid decay property extend to this generalized setting, and give some examples of groupoids admitting (RD).  
Aside from groups, all examples of groupoids with the (RD) property satisfy the polynomial growth condition, defined below.  
In the present work, we show that for a large class of \'etale groupoids, this is about as much as one can expect.

The goal of the present work is to introduce a rapid decay type property for twisted \'etale groupoids, as introduced in \cite{Rena2008}.  
This simultaneously generalizes the results of \cite{Hou2017} and the appendix of \cite{Math2006}.  
We show that, under mild topological assumptions, principality conditions on the \'etale groupoid imply that rapid decay (with or without twists) is equivalent to polynomial growth of the length function (see Theorems \ref{princRD} and \ref{topprincRD}).

The rest of the paper is organized as follows.  
In section 2, we provide some background information on groupoids, twists, and their operator algebras, then define property (RD) and list some known consequences.   
Section 3 contains the main results of the paper, where we show that for continuous length functions and principal groupoids, property (RD) is equivalent to polynomial growth.  
The last section is devoted to studying some permanence properties of property (RD).

\section{Preliminaries}

\subsection{Groupoids, twists, and their algebras}

In this section we recall the definitions of groupoids, twists, and the (reduced) $C^*$-algebras associated to such data.  
For additional background information, we refer the reader to \cite{Rena1980,will2019,sims2020}.

By a \textit{groupoid} we mean a small category in which every morphism is invertible.  
We will typically denote groupoids by the calligraphic letters $\G$ and $\H$.  
Given a groupoid $\G$, the set of objects of $\G$ (considered a subset of $\G$ by identifying an object with its identity morphism) is called the \textit{unit space} of $\G$, and is denoted $\Go$.  
Additionally, we denote the source and range maps by $s,r:\G\to\Go$, respectively, we denote the inverse map $\G\to\G$ by $\gamma\mapsto\gamma^{-1}$, and consider composition in the category to be a map from $\Gt:=\{(\gamma,\sigma)\in\G\times\G:s(\gamma)=r(\sigma)\}$ to $\G$, and write it $(\gamma,\sigma)\mapsto\gamma\sigma$.  
Given two groupoids $\G$, $\H$, a {\em groupoid homomorphism} from $\G$ to $\H$ is a map $\phi:\G\to\H$ that is compatible with the source, range, product, and inversion maps.

Let $\G$ be a groupoid.  Given $x\in\Go$, the {\em source fiber} of $x$ is the subset $\G_x:=\{\gamma\in\G:s(\gamma)=x\}$, the \textit{range fiber} at $x$ is $\G^x:=\{\gamma\in\G:r(\gamma)=x\}$, and the \textit{isotropy group} at $x$ is $\G_x^x=\G_x\cap\G^x$.  
The \textit{isotropy subgroupoid} of $\G$ is the subgroupoid $\operatorname{Iso}(\G)=\sqcup_{x\in\Go}\G_x^x$.  
Note that there are inclusions $\Go\subset\operatorname{Iso}(\G)\subset\G$.
We say that $\G$ is a \textit{group bundle} if $\operatorname{Iso}(\G)=\G$, and we say that $\G$ is \textit{principal} if $\operatorname{Iso}(\G)=\Go$.

In this paper, a \textit{topological groupoid} is a groupoid $\G$ equipped with a locally compact and Hausdorff topology such that all the structure maps are continuous, where the domain of the composition map, $\Gt$, is given the relative product topology.\footnote{Our assumptions that the topology be locally compact and Hausdorff are common, although not universal.}
A topological groupoid is said to be \textit{\'etale} if the source map is a local homeomorphism.  
A subset $K$ of $\G$ is a \textit{bisection} if it is contained in an open set $U$ of $\G$ such that the restrictions of the source and range maps to $U$ are homeomorphisms onto open subsets of $\Go$.

We now list some of the basic facts about \'etale groupoids which will be used in the sequel.  
For proofs, one can consult the references listed at the beginning of this section.

\begin{proposition}
    Let $\G$ be an \'etale groupoid.
    \begin{enumerate}[label = (\roman*)]
        \item The unit space $\Go$ is a clopen subset of $\G$.
        \item For each $x\in\Go$, the source and range fibres $\G_x$ and $\G^x$ are discrete subspaces of $\G$.
        \item The collection of open bisections of $\G$ forms a basis for the topology of $\G$.
        \item The product map $\Gt\to\G$, $(\gamma,\sigma)\mapsto\gamma\sigma$ is an open map.
    \end{enumerate}
\end{proposition}

A rich source of examples for groupoids come from the notion of an action of a groupoid on a space.  
Before defining groupoid actions, we recall the notion of a  fibered product:  
If $Y_1,Y_2,Z$ are sets and $f_i:Y_i\to Z$ are surjective functions, the {\em fibered product of $Y_1$ and $Y_2$ over $Z$ relative to the maps $f_1,f_2$} is the space
\begin{align*}
    Y_1\tensor[_{f_1}]{*}{_{f_2}}Y_2:=\{(y_1,y_2)\in Y_1\times Y_2: f_1(y_1)=f_2(y_2)\}.
\end{align*}
When $Y_1,Y_2,Z$ are topological spaces and $f_1$ and $f_2$ are continuous, we endow  $Y_1\tensor[_{f_1}]{*}{_{f_2}}Y_2$ with the relative product topology it inherits from being a subspace of the product space $Y_1\times Y_2$.

\begin{definition}
    Let $\G$ be a groupoid and let $Y$ be a set.  
    An {\em action of $\G$ on $Y$}\footnote{What we define here is sometimes called a {\em left action}.  Right actions are defined analogously, but there is no need to consider them here, so we omit the details.} is the data of a surjective map $p:Y\to\Go$, called the {\em anchor map} for the action, and a map
    \begin{align*}
        \G\tensor[_s]{*}{_p}Y\to Y, \qquad 
        (\gamma,y)\mapsto \gamma\cdot y,
    \end{align*}
    such that the following conditions are satisfied:
    \begin{itemize}
        \item If $y\in Y$ and $(\gamma,y)\in\G\tensor[_s]{*}{_p}Y$, then $p(\gamma\cdot y)=r(\gamma)$ and $p(y)\cdot y=y$.
        \item If $(\eta,y)\in\G\tensor[_s]{*}{_p}Y$ and $(\gamma,\eta)\in\Gt$, then $\gamma\cdot(\eta\cdot y)=(\gamma\eta)\cdot y$.
    \end{itemize}
    When $\G$ is a topological groupoid $Y$ is a locally compact Hausdorff space, an action of $\G$ on $Y$ is said to be {\em continuous} if the anchor map $p:Y\to\Go$ and the product map $\G\tensor[_s]{*}{_p}Y\to Y$ are continuous.
\end{definition}

Let $\G$ be a groupoid, let $Y$ be a set, and suppose $\G$ acts on $Y$ with anchor map $p$.  
We define the {\em transformation groupoid}, denoted $\G\ltimes Y$, associated to this action, as follows:  
as a set, $\G\ltimes Y=\G\tensor[_s]{*}{_p}Y$.  
The source, range, and inverse maps are given, for $(\gamma,y)\in\G\ltimes Y$, as follows:
\begin{align*}
    s(\gamma,y)&=(p(y),y),\\
    r(\gamma,y)&=(p(\gamma\cdot y),\gamma\cdot y),\\
    (\gamma,y)^{-1}&=(\gamma^{-1},\gamma\cdot y).
\end{align*}
The product in $G\ltimes Y$ is defined as follows:  if $((\gamma,y),(\eta,z))\in(\G\ltimes Y)^{(2)}$, then $y=\eta\cdot z$ and 
\begin{align*}
    (\gamma,\eta\cdot z)(\eta,z)=(\gamma\eta,z).
\end{align*}
When $\G$ is a topological groupoid, $Y$ is a locally compact Hausdorff space, and the action is continuous, $\G\ltimes Y$ is a topological groupoid.  Moreover, when $\G$ is \'etale, so is $\G\ltimes Y$.

We now lay out our notation for the various vector spaces and algebras associated to groupoids we use in our analysis.  
Fix an \'etale groupoid $\G$.  
The space $\CcG$ of continuous and compactly supported functions $\G\to\C$ is \textit{a priori} a vector space.  
We give it the structure of a $*$-algebra, with product and involution given by the formulas 
\begin{align*}
    (f*g)(\gamma)&=\sum_{\alpha\beta=\gamma}f(\alpha)g(\beta)\\
    f^*(\gamma)&=\overline{f(\gamma)}
\end{align*}
for $f,g\in\CcG$ and $\gamma\in\G$.\footnote{Proving that $f*g$ defined above lies in $\CcG$ makes use of the \'etale condition.}

For $f\in\CcG$, the $\sup$-norm is denoted $\|f\|_\infty=\sup_{\gamma\in\G}|f(\gamma)|$.  
While this is a $C^*$-norm for the pointwise operations, it fails to be submultiplicative for the product and involution defined above.  
In order to define a $C^*$-norm on $\CcG$, we look at a natural class of representations of this algebra on Hilbert spaces.  
For each $x\in\Go$, let $\C\G_x$ denote the space of functions $\G_x\to\C$ with finite support, and let $\ell^2\G_x$ denote the Hilbert space of square summable functions $\G_x\to\C$.  
We define a representation $\lambda_x:\CcG\to\bbB(\ell^2\G_x)$, called the \textit{(left) regular representation} at $x$ as follows:  for $f\in\CcG$, the operator $\lambda_x(f)\in\bbB(\ell^2\G_x)$ acts on $\xi\in\ell^2\G_x$ via the formula
\begin{align*}
	[\lambda_x(f)\xi](\gamma)=\sum_{\eta\in\G_x}f(\gamma\eta^{-1})\xi(\eta)
\end{align*}
for all $\gamma\in\G_x$.  We then define the {\em reduced $C^*$-norm} on $\CcG$ by the formula
\begin{align*}
	\|f\|_{C^*_r\G}=\sup_{x\in\Go}\|\lambda_x(f)\|_{\bbB(\ell^2\G_x)}.
\end{align*}
The {\em reduced $C^*$-algebra} of $\G$, denoted $C^*_r\G$, is then the completion of $\CcG$ with respect to this norm.

We now proceed to define twists over \'etale groupoids.  Our notation will follow that of \cite{sims2020}.

\begin{definition}\label{deftwist}
    Let $\G$ be an \'etale groupoid.  
    By a {\em twist} over $\G$ we mean a sequence
    \begin{center}
        \begin{tikzcd}
            \mathcal{G}^{(0)}\times\mathbb{T} \arrow[r, "i"] & \mathcal{E} \arrow[r, "\pi"] & \mathcal{G}
        \end{tikzcd}
    \end{center}
    of topological groupoids, where $\Go\times\T$ is considered as a trivial group bundle over $\Go$, and $i$ and $\pi$ are continuous groupoid homomorphisms which restrict to homeomorphisms on unit spaces (we identify $\E^{(0)}$ with $\Go$ via $\pi$), such that
    \begin{itemize}
        \item the sequence is short-exact, meaning that $i$ is injective, $\pi^{-1}(\Go)=i(\Go\times\T)$, and $\pi$ is surjective,
        \item for all $\ve\in\E$ and $z\in\T$, we have $i(r(\ve),z)\ve=\ve i(s(\ve),z)$, and
        \item every $\gamma\in\G$ admits an open neighborhood $U\subset\G$ and a continuous section $S:U\to\E$ for the map $\pi$ (meaning $\pi\circ S=\id_U$), such that the map $U\times\T\to\pi^{-1}(U)$ given by $(\eta,z)\mapsto i(r(\eta),z)S(\eta)$ is a homeomorphism.
    \end{itemize}
\end{definition}

The second condition is often seen as requiring that the image of $i$ is ``central" in $\E$, and the third condition means that we can view the map $\pi$ as a ``locally trivial $\G$-bundle."
As we have no need to consider multiple twists over the same groupoid, in the sequel we shall simply refer to the groupoid $\E$ as a ``twist" over $\G$, leaving the maps $i$ and $\pi$ implicit, or we shall say ``let $\E\overset{\pi}{\to}\G$ be a twist", if the bundle map $\pi$ need be made explicit.

Let $\E$ be a twist over an \'etale groupoid $\G$.  
For $\ve\in\E$ and $z\in\T$, we denote by $z\cdot\ve$ the element of $\E$ given by $i(r(\ve),z)\ve$.
Similarly, let $\ve\cdot z=\ve i(s(\ve),z)$, so that $z\cdot\ve=\ve\cdot z$ by the second condition in Definition \ref{deftwist}.  
If $\ve_1,\ve_2\in\E$ and $\pi(\ve_1)=\pi(\ve_2)$, then by {\cite[Lemma 11.1.3]{sims2020}} there is a unique $[\ve_1,\ve_2]\in\T$ such that $\ve_1=[\ve_1,\ve_2]\cdot\ve_2$.

Let $\ScGE=\{f\in C_c(\E):f(z\cdot\ve)=zf(\ve)\text{ for all }z\in\T,\ve\in\E\}$.  
With pointwise addition and scalar multiplication, this is a $\C$-vector space.  
It is a $*$-vector space, with involution given by $(f^*)(\ve)=\overline{f(\ve^{-1})}$ for $f\in\ScGE$ and $\ve\in\E$.  
To define a multiplication on $\ScGE$, fix a (not necessarily continuous) section $\rho:\G\to\E$ for the map $\pi$, and for $f,g\in\ScGE$ define $f*g\in\ScGE$ by
\begin{align*}
    (f*g)(\ve)=\sum_{\gamma\in\G_{s(\ve)}}f(\ve\rho(\gamma)^{-1})g(\rho(\gamma)).
\end{align*}
By the $\T$-equivariance of functions in $\ScGE$, the above formula is independent of the chosen section $\rho$.

For each $x\in\Go$, define a representation $\lambda_x^\rho$ of $\ScGE$ on $\ell^2\G_x$ by extension of the above convolution formula:
\begin{align*}
    [\lambda_x^\rho(f)\xi](\gamma)
    =\sum_{\eta\in\G_x}f\left(\rho(\gamma)\rho(\eta)^{-1}\right)\xi(\eta).
\end{align*}
Up to unitary equivalence, this representation is independent of the chosen section $\rho$.  
We define $C^*_r(\G,\E)$ to be the completion of $\ScGE$ with respect to the norm
\begin{align*}
    \|f\|_{C^*_r(\G,\E)}=\sup_{x\in\Go}\|\lambda_x^\rho(f)\|_{\bbB(\ell^2\G_x)}.
\end{align*}

\begin{example}\label{trivtwist}
    Let $\G$ be an \'etale groupoid.  
    If we consider $\G\times\T$ as a topological groupoid, with the product topology and pointwise operations, then 
    \begin{center}
        \begin{tikzcd}
            \mathcal{G}^{(0)}\times\mathbb{T} \arrow[r, "i"] & \mathcal{G}\times\mathbb{T} \arrow[r, "\pi"] & \mathcal{G}
        \end{tikzcd}
    \end{center}
    where $i$ is the inclusion map and $\pi(\gamma,z)=\gamma$, defines a twist over $\G$, called the {\em trivial twist} over $\G$.  
    There is a natural identification $\Sigma_c(\G,\G\times\T)$ with $\CcG$, sending $f\in\Sigma_c(\G,\G\times\T)$ to the map $\G\ni\gamma\mapsto f(\gamma,1)\in\C$, which extends to an isomorphism from $C^*_r(\G,\G\times\T)$ onto $C^*_r\G$.
\end{example}

\subsection{Rapid decay for twisted \'etale groupoids}

In this subsection, we describe the basic properties of length functions on groupoids, and give a definition of property (RD) for twisted \'etale groupoids.  We begin by recalling the definition of a length function on $\G$, as given in \cite{Hou2017}.

\begin{definition}
    Let $\G$ be a groupoid.  By a \textit{length function} on $\G$ be we mean a map $\ell:\G\to[0,\infty)$ satisfying the following conditions:
    \begin{itemize}
        \item $\ell(x)=0$ for any $x\in\Go$,
        \item $\ell(\gamma^{-1})=\ell(\gamma)$ for any $\gamma\in\G$, and
        \item $\ell(\gamma\eta)\leq\ell(\gamma)+\ell(\eta)$ for any $(\gamma,\eta)\in\Gt$.
    \end{itemize}
\end{definition}

Now suppose that $\G$ is a topological groupoid.  
We say that the length function $\ell$ is \textit{continuous} if it is continuous as a map from $\G$ to $[0,\infty)$.  
A weaker condition that one can ask for is that $\ell$ be \textit{locally bounded}, meaning that $\sup_{\gamma\in K}\ell(\gamma)$ is finite for all compact sets $K$.  
As a partial converse to this definition, following \cite{MaWu2020} we say that $\ell$ is \textit{proper} if for every for every subset $K\subset\G\setminus\Go$, finiteness of the quantity $\sup_{\gamma\in K}\ell(\gamma)$ implies that $K$ is pre-compact.

\begin{example} \ 
    \begin{enumerate}[label = (\arabic*)]
        \item Suppose $\Gamma$ is a discrete group, $\ell$ is a length function on $\Gamma$, and suppose that $\Gamma$ acts by homeomorphisms on the locally compact Hausdorff space $X$.  
        On the transformation groupoid $\Gamma\ltimes X$ one can define a length function (still denoted $\ell$) by the fomula $\ell(\gamma,x)=\ell(\gamma)$.  
        This induced length function is continuous, and proper if the length on $\Gamma$ is proper and $X$ is assumed to be compact.  
        \item More generally, suppose $\G$ and $\H$ are groupoids, and that $\phi:\H\to\G$ is a groupoid homomorphism.  If $\ell$ is a length function on $\G$, then the formula $(\phi^*\ell)(\eta)=\ell(\phi(\eta))$ defines a length function $\H$.  
        \item Let $\G$ be an \'etale groupoid, and assume that $\G$ is compactly generated, meaning that there is a pre-compact subset $K\subset\G$ such that every $\gamma\in\G$ can be written as a product of elements of $K\cup K^{-1}$.  
        Given such a $K$, we define a length $\ell$ on $\G$ by $\ell(x)=0$ for $x\in\Go$, and by
        \begin{align*}
            \ell(\gamma)=\min\{n\in\N:\gamma=\gamma_1\cdots\gamma_n\text{ for some }\gamma_k\in K\cup K^{-1}\}
        \end{align*}
        for $\gamma\in\G\setminus\Go$.
    \end{enumerate}
\end{example}

Given a length function $\ell$ on a groupoid $\G$, for each $x\in\Go$ one can define a pseudometric $\rho_{\ell,x}$ on the source fibre $\G_x$ by the formula $\rho_{\ell,x}(\gamma_1,\gamma_2)=\ell(\gamma_1\gamma_2^{-1})$.  
For $\gamma\in\G_x$, $r>0$, the {\em closed ball of radius $r$ centered at $\gamma$} with respect to this metric will be denoted $B_\ell(\gamma,r)=\{\eta\in\G_x:\ell(\gamma\eta^{-1})\le r\}$.
In section 4 of \cite{MaWu2020}, they study the geometric structure a length function imposes on an \'etale groupoid.  
One particularly nice result, which we will use, is their \textit{local slice lemma}, which we repeat below for convenience.

\begin{lemma}[{\cite[Lemma 5.10]{MaWu2020}}] \label{LocSliceLemma}
    Let $\G$ be a $\sigma$-compact, \'etale groupoid, and let $\ell$ be a continuous and proper length function on $\G$.  
    For every $x\in\Go$ and every pair of constants $R,\ve>0$, there exist a number $R'\in[R,R+\ve)$, an open neighborhood $V\subset\Go$ of $x$, an open subset $W$ of $\G$, and a homeomorphism $\Phi:B_\ell(x,R')\times V\to W$, satisfying the following conditions:
    \begin{enumerate}[label = (\roman*)]
        \item $\Phi(x,y)=y$ for any $y\in V$,
        \item $\Phi(\gamma,x)=\gamma$ for any $\gamma\in B(x,R')$,
        \item $\Phi(B_\ell(x,R')\times\{y\})=B_\ell(y,R')$ for every $y\in V$, and
        \item $|\ell(\gamma\eta^{-1})-\ell(\Phi(\gamma,y)\Phi(\eta,y)^{-1})|<\ve$ for all $\gamma,\eta\in B(x,R')$ and all $y\in V$.
    \end{enumerate}
\end{lemma}

This lemma is particularly useful, as it allows us to translate data from a finite subset of source fibres to nearby source fibres.  
We provide one modest example.

\begin{corollary}\label{SameCards}
    Let $\G$ be a $\sigma$-compact, \'etale groupoid, and let $\ell$ be a continuous and proper length function on $\G$.  
    For every $x_0\in\Go$ and $R>0$, there is some open neighborhood $V\subset\Go$ of $x_0$ such that $|B_\ell(x,R)|=|B_\ell(x_0,R)|$ for all $x\in V$.
\end{corollary}

We now define some seminorms on $C_c(\G)$ which will be relevant for our discussion of rapid decay type properties for $\G$.

\begin{definition}\label{GrpdRDnorms} 
    Let $\G$ be an \'etale groupoid, let $\E\overset{\pi}{\to}\G$ be a twist over $\G$, and let $\ell$ be a length function on $\G$.  
    Fix a section $\rho:\G\to\E$ for the map $\pi$.
    For each $x\in\Go$, and $t\ge0$, define seminorms on $\ScGE$ by
    \begin{align*}
        &\|f\|_{\E,\ell,t,s,x}=\left(\sum_{\gamma\in\G_x}|f(\rho(\gamma))|^2(1+\ell(\gamma)^{2t}\right)^{1/2}, \\
        &\|f\|_{\E,\ell,t,s}=\sup_{x\in\Go}\|f\|_{\E,\ell,t,s,x},\\
        &\|f\|_{\E,\ell,t}=\max\{\|f\|_{\E,\ell,t,s},\|f^*\|_{\E,\ell,t,s}\}.
    \end{align*}
\end{definition}

\begin{definition}\label{RDgrpd}
    Let $\G$ be an \'etale groupoid, let $\E$ be at twist over $\G$, and let $\ell$ be a length function on $\G$.  
    We say that $\G$ has \textit{$\E$-twisted rapid decay} (or $\E$-(RD) for short) with respect to the length function $\ell$ if there exist constants $C,t\ge0$ such that 
    \begin{align*}
        \|f\|_{C^*_r(\G,\E)}\le C\|f\|_{\E,\ell,t}
    \end{align*}
    for all $f\in\ScGE$.
    In the case that $\E=\G\times\T$ is the trivial twist, we identify $\ScGE$ with $\CcG$, remove the $\E$ in the subscript of the norms, and simply say that $\G$ has the \textit{rapid decay property}, or property (RD), when it has $\E$-(RD) for this twist.
\end{definition}

Immediately from the definition, one can see that this generalizes the notion of rapid decay for discrete groups:  if $\Gamma$ is a discrete group with a length function $\ell$, then $\Gamma$ has property (RD) with respect to $\ell$ as in Definition \ref{RDgrpd} if and only if it satisfies property (RD) with respect to $\ell$ as described in the first paragraph of the present work.

We now proceed to show that, for a fixed length function, (RD) implies $\E$-(RD) for all twists $\E$.  
This has been shown in the group case, but in order to adapt the proof to this setting, we need a lemma.

\begin{lemma}
    If $f\in\ScGE$, then the map $\G\to\C$, $\gamma\mapsto|f(\rho(\gamma))|$, belongs to $\CcG$.
\end{lemma}
\begin{proof}
    First, we show the map has compact support.  
    Observe that if $\gamma\in\G$ and $|f(\rho(\gamma))|\ne0$, then $\rho(\gamma)\in\supp(f)$.  
    Thus $\gamma=\pi(\rho(\gamma))\in\pi(\supp(f))$.  
    As the latter set is compact, it follows that the set
    \begin{align*}
        \{\gamma\in\G:|f(\rho(\gamma))|\ne0\}
    \end{align*}
    has compact closure.  
    Next, we show continuity.  
    Fix $\gamma_0\in\G$ and $\ve>0$.  
    There is an open neighborhood $U\subset\G$ of $\gamma_0$ and a continuous section $\rho_U:U\to\E$ of $\pi$.  
    As $f$ is continuous, there is an open neighborhood $V\subset\E$ of $\rho_U(\gamma_0)$ such that $|f(\rho_U(\gamma_0))-f(\delta)|<\ve$ whenever $\delta\in V$.  
    Letting $U_0=\rho_U^{-1}(V)\subset\G$, we see that $U_0$ is an open neighborhood of $\gamma_0$, and if $\gamma\in U_0$, we have 
    \begin{align*}
        \left||f(\rho(\gamma_0))|-|f(\rho(\gamma))|\right|
        &=\left||f(\rho_U(\gamma_0))|-|f(\rho_U(\gamma))|\right|
        \\&\le|f(\rho_U(\gamma_0))-f(\rho_U(\gamma))|
        <\ve,
    \end{align*}
    where the first equality follows from the fact that $f(z\cdot\delta)=zf(\delta)$ for all $z\in\T$ and $\delta\in\E$.  
\\\end{proof}

\begin{proposition}\label{RDiERD}
    Let $\G$ be an \'etale groupoid, and let $\ell$ be a length function on $\G$.  If $\G$ has (RD) with respect to $\ell$, then it has $\E$-(RD) with respect to $\ell$ for any twist $\E$ over $\G$.  
\end{proposition}
\begin{proof}
    With the above lemma, we can proceed as in the proof of Lemma 6.7 in \cite{Math2006}.  
    Let $\E\overset{\pi}{\to}\G$ be a twist over $\G$, and let $\rho:\G\to\E$ be a section for $\pi$. 
    As $\G$ has property (RD), there are constants $C,t\ge0$ such that $\|g\|_{C^*_r\G}\le C\|g\|_{\ell,t}$ for all $f\in\CcG$. 
    Fix $f\in\ScGE$, and define $g\in\CcG$ by $g(\gamma)=|f(\rho(\gamma))|$.  
    If $x\in\Go$ and $\xi\in\ell^2\G_x$, then for any $\gamma\in\G_x$ we have 
    \begin{align*}
        [\lambda_x^\rho(f)\xi](\gamma)|
        \le\sum_{\eta\in\G_x}|f(\rho(\gamma\eta^{-1}))||\xi(\eta)|
        =\left[\lambda_x(g)|\xi|\right](\gamma).
    \end{align*}
    Summing over $\gamma\in\G_x$ and taking a square root yields
    \begin{align*}
        \|\lambda_x^\rho(f)\xi\|_{\ell^2\G_x}
        &\le\|\lambda_x(g)|\xi|\|_{\ell^2\G_x}
        \le\|g\|_{C^*_r\G}\|\xi\|_{\ell\G_x}
        \\&\le C\|g\|_{\ell,t}\|\xi\|_{\ell^2\G_x}
        =C\|f\|_{\E,\ell,t}\|\xi\|_{\ell^2\G_x}.
    \end{align*}    
    Taking suprema, the above inequality implies $\|f\|_{C^*_r(\G,\E)}\le C\|f\|_{\E,\ell,t}$.
\\\end{proof}

As mentioned in the introduction, one of the main motivations for studying rapid decay in the group setting was that it yields nice a nice subalgebra of the reduced group $C^*$-algebra.  
This was also shown to be the case for groupoids with (RD) in \cite{Hou2017}, and in the twisted setting, we obtain a similar result.
Let $S_\ell(\G,\E)$ denote the completion of $\ScGE$ with respect to the topology induced by the family of norms $\{\|\cdot\|_\infty\}\cup\{\|\cdot\|_{\E,\ell,t}:t\in\Z_{\ge0}\}$.  
With minor modifications, the proofs of Lemma 3.3 and Proposition 3.4 in \cite{Hou2017} yield the following result:

\begin{proposition}
    Let $\G$ be an \'etale groupoid, let $\ell$ be a length function on $\G$, and let $\E$ be a twist over $\G$.  
    Then $\G$ has $\E$-(RD) with respect to $\ell$ if and only if $S_\ell(\G,\E)\subset\CrGE$. 
    Moreover, if the length function $\ell$ is continuous, and $\G$ has $\E$-(RD) with respect to $\ell$, then $S_\ell(\G,\E)$ is a dense, Fr\'echet $*$-subalgebra of $\CrGE$. 
\end{proposition}

The proof of Theorem 4.2 in \cite{Hou2017} also generalizes to the this setting, and we obtain the following result.

\begin{proposition}
    Let $\G$ be an \'etale groupoid, let $\E$ be a twist over $\G$, and let $\ell$ be a continuous length function on $\G$.  
    If $\G$ has $\E$-(RD) with respect to $\ell$, then $S_\ell(\G,\E)$ is an inverse closed subalgebra of $\CrGE$, and the inclusion induces an isomorphism at the level of $K$-theory.
\end{proposition}

We end this section by giving another class of groupoids with length functions for which property (RD) holds.  
Let $\G$ be an \'etale groupoid, and let $\ell$ be a length on $\G$.  
We say that $\G$ has {\em polynomial growth} with respect to $\ell$ if there exist constants $C,n>0$ such that $|B_\ell(x,r)|\le C(1+r)^n$ for all $x\in\Go$ and $r>0$.

\begin{proposition}[{\cite[Proposition 3.5]{Hou2017}}]
    Let $\G$ be an \'etale groupoid, and let $\ell$ be a length function on $\G$.  
    If $\G$ has polynomial growth with respect to $\ell$, then $\G$ has property (RD) with respect to $\ell$, and hence has $\E$-(RD) for all twists $\E$ over $\G$.
\end{proposition}

In \cite{Hou2017}, all examples of groupoids with rapid decay had polynomial growth with respect to the given length function.  
In the next section, we shall see that under the assumption that the groupoid is principal, this is about as much as one can expect.

\section{Principal Groupoids}

Recall that a groupoid $\G$ is called principal if $\operatorname{Iso}(\G)=\Go$, or equivalently, if the map $s\times r:\G\to\Go\times\Go$ is injective. 
As a topological analogue of principality, a topological groupoid $\G$ is said to be \textit{topologically principal} if the set of units $x\in\Go$ such that $\G_x^x=\{x\}$ is dense in $\Go$.  
In this section, we show that (topologically) principal groupoids admit property (RD) only when the length function has polynomial growth, and some continuity conditions.  
This generalizes some known results, see \cite{Chen2003}, \cite{JiYu2020}.  
Our strategy is inspired by the proof of Theorem 2.1 of \cite{Chen2003}, but details require attention in this more general setting.

\begin{lemma}\label{mainlemma}
    Let $\G$ be an \'etale groupoid, and let $\ell$ be a continuous length function on $\G$. 
    Suppose that for every $C>0$, $D>0$, there exist $R>0$, $x\in\Go$, and a finite set $F\subset B(x,R)$, such that 
    \begin{enumerate}[label = (\roman*)]
        \item $|F|>C(1+R)^D$, and
        \item the restriction of the range map to $F$ is injective,
    \end{enumerate}
    Then $\G$ does not have $\E$-(RD) with respect to $\ell$ for any twist $\E$ over $\G$. 
\end{lemma}
\begin{proof}
    Fix $C,t>0$.  
    By assumption, there exist $R>0$, $x_0\in\Go$, and a finite set $F\subset B(x_0,r)$ such that $r\mid_F$ is injective and 
    $$|F|>C^2(1+R)^{6t}.$$
    Without loss of generality, we may assume that $x_0\in F$. 
    For each $\gamma\in F$, fix a pre-compact open bisection neighborhood $W_\gamma$ of $\gamma$ such that 
    \begin{itemize}
        \item $W_{x_0}\subset\Go$,
        \item $s(W_\gamma)\subset W_{x_0}$ for all $\gamma\in F$, and 
        \item the collection $\{r(W_\gamma):\gamma\in F\}$ of subsets of $\Go$ is pairwise disjoint. 
    \end{itemize}
    Put $Z=FF^{-1}$.  
    For each $g\in Z$, there exist (unique) $\gamma_1,\gamma_2\in F$ such that $g=\gamma_1\gamma_2^{-1}$.  
    Define an open neighborhood $V_g$ of $g$ as follows:  If $\gamma_1=\gamma_2$ we set $V_g=r(W_{\gamma_1})$, and otherwise $V_g=W_{\gamma_1}W_{\gamma_2}^{-1}$.
    Now fix a pre-compact open bisection neighborhood $U_g$ of $g$ such that $\overline U_g\subset V_g$, and such that $|\ell(\gamma)-\ell(g)|<R$ for all $\gamma\in U_g$.  
    \\\indent Let $\E\overset{\pi}{\to}\G$ be a twist, and let $\rho:\G\to\E$ be a (not necessarily continuous) section for $\pi$.  
    Then we may assume that for each $g\in Z$, the open bisection $U_g$ admits a continuous section $\rho_g:\G\to\E$ for $\pi$, such that the map
    \begin{align*}
        \Psi_g:U_g\times\T\to\pi^{-1}(U_g),\qquad(\gamma,z)\mapsto z\cdot\rho_g(\gamma)
    \end{align*}
    is a homeomorphism.
    For each $g\in Z$, fix a function $\tilde{f}_g\in\CcG$ such that $0\le\tilde{f}_g(\gamma)\le1$ for all $\gamma\in\G$, with $\tilde{f}_g(g)=1$ and $\supp(\tilde{f}_g)\subset U_g$.
    Define $f_g\in\ScGE$ for $g\in Z$ to be the function such that $\supp(f_g)\subset\pi^{-1}(U_g)$, and 
    \begin{align*}
        f_g(\rho_g(\gamma))
        =[\rho_g(g),\rho(\gamma_1)\rho(\gamma_2)^{-1}]\tilde{f}_g(\gamma),
    \end{align*}
    whenever $\gamma\in U_g$, where $\gamma_1$ and $\gamma_2$ are the (unique) elements of $F$ such that $g=\gamma_1\gamma_2^{-1}$.  
    Now let $\xi=|F|^{-1/2}\chi_F\in\ell^2\G_{x_0}$, where $\chi_F$ denotes the indicator function for the set $F\subset\G_{x_0}$.  
    If $\gamma\in\G_{x_0}$, then $[\lambda_x^\rho(f)\xi](\gamma)=0$ unless $\gamma\in\F$.  
    In this case, we set $Z^\gamma=Z\cap\G^{r(\gamma)}$, and we have 
    \begin{align*}
        [\lambda_x^\rho(f)\xi](\gamma)
        &=|F|^{-1/2}\sum_{\eta\in F}f(\rho(\gamma)\rho(\eta)^{-1})\\
        &=|F|^{-1/2}\sum_{g\in Z^\gamma}f_g(\rho(\gamma)\rho(g^{-1}\gamma)^{-1})\\
        &=|F|^{-1/2}\sum_{g\in Z^\gamma}\tilde{f}_g(g)\\
        &=|F|^{1/2}.
    \end{align*}
    Squaring and summing over $\gamma\in\G_{x_0}$, we obtain $\|\lambda_x^\rho(f)\xi\|_{\ell^2\G_{x_0}}^2=|F|^2$, so $\|f\|_{C^*_r(\G,\E)}\ge|F|$.
    \\\indent Next we estimate $\|f\|_{\E,\ell,t}$. 
    We claim that, for all $x\in\Go$ we have
    \begin{align*}
        \left|\bigcup_{g\in Z}\G_x\cap U_g\right|\leq|F|
    \end{align*}
    To see this, first note that if $g\in Z$, then $s(U_g)\subset r(W_\eta)$ for some $\eta\in F$.  
    Hence if $x\in\Go$ and $x\notin r(W_\eta)$ for all $\eta\in F$, then $\bigcup_{g\in Z}\G_x\cap U_g=\varnothing$, and the estimate holds.  
    Now suppose $x\in r(W_\eta)$ for some (necessarily unique) $\eta\in F$.  
    If $g\in Z$ and $\G_x\cap U_g$ is nonempty, then $g=\gamma\eta^{-1}$ for some $\gamma\in F$, so
    \begin{align*}
        \left(\bigcup_{g\in Z}\G_x\cap U_g\right)\subset\left(\bigcup_{\gamma\in F}\G_x\cap U_{\gamma\eta^{-1}}\right).
    \end{align*}
    As each $U_g$ is a bisection,  the cardinality of the set on the right-hand side is clearly no more than $|F|$.
    This proves the claim.   
    Now for $x\in\Go$, the claim implies 
    \begin{align*}
        \|f\|_{\E,\ell,t,s,x}^2
        &=\sum_{\gamma\in\G_x}|f(\gamma)|^2(1+\ell(\gamma))^{2t}\\
        &\leq\sum_{\gamma\in\cup_{g\in Z}(\G_x\cap U_g)}(1+\ell(\gamma))^{2t}\\
        &<|F|(1+2R)^{2t}.
    \end{align*}    
    Similarly, we obtain $\|f^*\|^2_{\E,\ell,t,s,x}<|F|(1+2R)^{2t}$, and it follows that 
    \begin{align*}
        \|f\|_{\E,\ell,t}\le|F|^{1/2}(1+2R)^t.
    \end{align*}
    Combining our estimates, we obtain
    \begin{align*}
        \|f\|_{C^*_r(\G,\E)}
        \ge|F|
        >C|F|^{1/2}(1+R)^{2t}
        \ge C|F|^{1/2}(1+2R)^{t}\
        \ge C\cdot\|f\|_{\E,\ell,t}.
    \end{align*}
    Since $C,t>0$ were arbitrary, it follows that $\G$ does not have $\E$-(RD) with respect to the length $\ell$.
\\\end{proof}

\begin{theorem}\label{princRD}
    Let $\G$ be a principal, \'etale groupoid, and let $\ell$ be a continuous length function on $\G$.  
    The following are equivalent:
    \begin{enumerate}
        \item $\G$ has polynomial growth with respect to $\ell$.
        \item $\G$ has $\E$-(RD) with respect to $\ell$ for all twists $\E$ over $\G$.
        \item $\G$ has $\E$-(RD) with respect to $\ell$ for some twist $\E$ over $\G$.
    \end{enumerate}
\end{theorem}
\begin{proof}
    If $\G$ has polynomial growth, then by \cite[Proposition 3.5]{Hou2017} $\G$ has (RD) with respect to $\ell$.  
    Proposition \ref{RDiERD} now implies that $\G$ has $\E$-(RD) with respect to $\ell$ for any twist $\E$ over $\G$.  
    Thus (1) implies (2).  
    Obviously, (2) implies (3), so we focus on showing (3) implies (1). 
    We prove the contrapositive, so assume that $\G$ does not have polynomial growth with respect to $\ell$.  
    Then for each $C,D>0$ there exist $x\in\Go$ and $r>0$ such that $|B(x,r)|>C(1+r)^D$.  Letting $F=B(x,r)$, we see that $F$ satisfies conditions (i) and (ii) of Lemma \ref{mainlemma}, and therefore $\G$ cannot have property $\E$-(RD) with respect to $\ell$ for any twist, and therefore (3) does not hold.
\\\end{proof}

A mild adjustment of our assumptions allows us to apply the local slice lemma, and we obtain the following.

\begin{theorem}\label{topprincRD}
    Let $\G$ be a $\sigma$-compact, topologically principal, \'etale groupoid, and let $\ell$ be a continuous and proper length function on $\G$.  
    The following are equivalent:
    \begin{enumerate}
        \item $\G$ has polynomial growth with respect to $\ell$.
        \item $\G$ has $\E$-(RD) with respect to $\ell$ for all twists $\E$ over $\G$.
        \item $\G$ has $\E$-(RD) with respect to $\ell$ for some twist $\E$ over $\G$.
    \end{enumerate}
\end{theorem}
\begin{proof}
    The only part that differs from the proof of Theorem \ref{princRD} is the proof of {\em (3)$\Rightarrow$(1)}.
    Suppose $\G$ does not have polynomial growth with respect to the length function $\ell$.  
    Fix $C,D>0$.  
    We can find some $x_0\in\Go$ and some $R_0>0$ such that 
    \begin{align*}
        |B(x_0,R_0)|>C(1+R_0)^D.
    \end{align*}
    Choose $\varepsilon>0$ such that 
    \begin{align*}
    |B(x_0,R_0)|>C(1+R_0+\varepsilon)^D,
    \end{align*}
    and moreover assume that $\varepsilon$ is chosen small enough so that $B(x_0,R_0+\varepsilon)=B(x_0,R_0)$.  
    By the local slice lemma, we can find some $R\in[R_0,R_0+\varepsilon)$ and some open set $V\subset\Go$ containing $x_0$, such that $|B(x,R)|=|B(x_0,R_0)|$ for all $x\in V$.  
    Since $\G$ is topologically principal, we can find some $x\in V$ with trivial isotropy.  
    Then
    \begin{align*}
    |B(x,R)|=|B(x_0,R_0)|>C(1+R_0+\varepsilon)^D>C(1+R)^D.
    \end{align*}
    Now $F=B(x,R)$ satisfies the conditions of Lemma \ref{mainlemma}, and therefore $\G$ cannot have property (RD).
\\\end{proof}

Now we consider the special case of coarse groupoids.  
As an application of the above result, we obtain a generalization of Theorem 2.1 from \cite{Chen2003}.  
We begin by briefly recalling their construction (for more details, one can consult \cite{STY2002,Roe2003}).  
Let $(X,d)$ be a discrete metric space, and assume for ease of exposition that it has \textit{bounded geometry}, meaning that for each $r>0$, there is a uniform bound on the cardinality of the balls $B(x,r)$ as $x$ varies over $X$. 
For $r\geq0$, let $E_r=\{(x,y)\in X\times X:d(x,y)\leq r\}$ denote the tube of radius $r$, and let $\overline{E_r}$ denote its closure in $\beta(X\times X)$, the Stone-\v Cech compactification of $X\times X$.  
As a space, the \textit{coarse groupoid} of $X$ is $\G(X)=\cup_{r\geq 0}\overline{E_r}\subset\beta X\times\beta X$.  
By Theorem 10.20 of \cite{Roe2003}, the pair groupoid structure on $X\times X$ extends continuously to $\G(X)$ , making it a principal, $\sigma$-compact, \'etale groupoid, with unit space homeomorphic to $\beta X$, and range and source maps respectively the unique extensions of of the first and second factor maps $X\times X\to X$.

We recall the notions necessary to define metric rapid decay for the bounded geometry metric space $(X,d)$.  
First, we recall the definition of the uniform Roe algebra associated to $X$.  
For a function $k:X\times X\to\C$, the \textit{propagation} of $k$ is the quantity $\propa(X)=\sup\{d(x,y):x,y\in X, k(x,y)\ne0\}$.  
$\C_u(X)$ be the space of those functions $k$ that are bounded (meaning $\sup\{|k(x,y)|:x,y\in X\}$ is finite) and of finite propagation.  
This is a $*$-algebra, which admits a canoncial action on the Hilbert space $\ell^2(X)$ of square-summable functions $X\to\C$.  
The $C^*$-algebra generated by $\C_u(X)$ is called the \textit{uniform Roe algebra} of $X$, and is denoted $C^*_u(X)$.

\begin{definition}
    Let $k:X\times X\to\mathbb C$ be given.  
    For $t\ge0$, we define the quantities $\|k\|_{BS,t},\|k\|_{BS^*,t}\in[0,\infty]$ by
    \begin{align*}
        &\|k\|_{BS,t}=\left(\sup_{y\in X}\sum_{x\in X}|k(x,y)|^2(1+d(x,y))^{2t}\right)^{1/2}\\
        &\|k\|_{BS^*,t}=\max\{\|k\|_{BS,t},\|k^*\|_{BS,t}\},
    \end{align*}
    where $k^*:X\times X\to\C$ is defined by $k^*(x,y)=\overline{k(y,x)}$.
    We denote by $BS_2(X)$ the space of all functions $k:X\times X\to\C$ such that $\|k\|_{BS^*,t}<\infty$ for all $t\ge0$.
\end{definition}

Note that $BS_2(X)$ is a Fr\'echet space with the topology given by the family of seminorms $\{\|\cdot\|_{BS^*,t}:t\in\Z_{\ge0}\}$.

\begin{definition}
    We say that $X$ has property (MRD), or has \textit{(metric) rapid decay}, if $BS_2(X)$ is contained in $C^*_u(X)$.\footnote{Our definition differs slightly from those given in \cite{Chen2003} and \cite{JiYu2020}, to account for the self-adjointness of the norms in Definition \ref{GrpdRDnorms}, which was adapted from \cite{Hou2017}.  This is only a matter of convention; and one can easily adapt these results to that setting.}
\end{definition}

We briefly outline the construction of a canonical length function on $\G(X)$ as done in Section 5 of \cite{MaWu2020}. 
First, observe that for each $r\geq 0$, the restriction of the metric $d:E_r\to[0,r]$ extends to a continuous map $\ell:\overline{E_r}\to[0,r]$, and that these extensions respect the inclusions $\overline{E_r}\subset\overline{E_{r'}}$ for $r'\geq r$, producing a well-defined length $\ell$ on $\G(X)$, which is clearly continuous.  
Properness of this length is also readily verified, as for each $r\geq0$ we have $\ell^{-1}([0,r])\subset\overline{E_r}$, a compact subset of $\G(X)$.

\begin{lemma}\label{mrdgrpdrd}
    Let $X$ be a discrete metric space with bounded geometry. 
    If $X$ has property $(MRD)$, then $\G(X)$ has property (RD) with respect to the length function defined above.
\end{lemma}
\begin{proof}
    Suppose $X$ has $(MRD)$.  
    The inclusion of $BS_2(X)$ into $C^*_u(X)$ is a closed map, hence continuous, and there exist $C,t\ge0$ such that $\|k\|_{C^*_u(X)}\leq C\|k\|_{BS^*,t}$ for all $k\in BS_2(X)$.  
    If now $f\in C_c(\G(X))$, let $k_f:X\times X\to\mathbb C$ denote the restriction of $f$ to $X\times X\subset\G(X)$.  
    Then we have the estimate
    \begin{align*}
        \|k_f\|_{BS,t}
        =\sup_{y\in X}\left(\sum_{x\in X}|f(x,y)|(1+\ell(x,y))^{2t}\right)^{1/2}
        =\sup_{y\in X}\|f\|_{\ell,t,s,y}
        \le\|f\|_{\ell,t,s}.
    \end{align*}
    Note that $(k_f)^*=k_{f^*}$, so $\|k_f\|_{BS^*,t}\le\|f\|_{\ell,t}$.  
    By Proposition 10.29 in \cite{Roe2003}, we have $\|f\|_{C^*_r\G}=\|k_f\|_{C^*_u(X)}$, and thus 
    \begin{align*}
        \|f\|_{C^*_r\G}\le C\|f\|_{\ell,t}.
    \end{align*}
\end{proof}

As in Definition 1.7 of \cite{Chen2003}, we say the metric space $(X,d)$ has \textit{polynomial growth} if there exist constants $C,n>0$ such that $|B(x,r)|\le C(1+r)^n$ for all $x\in X$ and $r\ge0$.  
An application of the local slice lemma yields the following result:

\begin{lemma}\label{mpggrpdpg}
    Let $X$ be a discrete metric space with bounded geometry. 
    Then $\G(X)$ has polynomial growth if and only if $X$ has polynomial growth.
\end{lemma}
\begin{proof}
    The forward implication is clear, so assume that $X$ has polynomial growth, and fix $C>0$, $d\in\mathbb N$ such that for all $R>0$, we have 
    \begin{align*}
        \sup_{x\in X}|B(x,R)|\leq C(1+R)^d.
    \end{align*}
    Fix an ultrafilter $\omega_0\in\beta X$ and $R>0$.  
    Choose $\ve>0$ such that $B_\ell(\omega,R+\ve)=B_\ell(\omega,R)$ for all $\omega\in\beta X$.  
    By Corollary \ref{SameCards}, there is an open neighborhood $V$ of $\omega_0$ in $\beta X$ such that $|B_\ell(\omega,R)|=|B_\ell(\omega_0,R)|$ for all $\omega\in V$.  
    Fixing some $x_0\in V\cap X$, we have 
    \begin{align*}
        |B_\ell(\omega_0,R)|=|B_\ell(x_0,R)|=|B(x_0,R)|\le C(1+R)^d.
    \end{align*}
    Therefore, $\G(X)$ has polynomial growth. 
\\\end{proof}

Combining Lemmas \ref{mpggrpdpg} and \ref{mrdgrpdrd}, we see that Theorem \ref{princRD} generalizes a previous theorem of Chen and Wei.

\begin{theorem}[{\cite[Theorem 2.1]{Chen2003}}]
     Let $X$ be a discrete metric space with bounded geometry.  
     Then $X$ has property $(MRD)$ if and only if $X$ has polynomial growth.  
\end{theorem}

\section{Permanence Results}

In this last section, we list some permanence properties enjoyed by the rapid decay property.  
We give conditions for products of (RD) groupoids to have (RD) (see Proposition \ref{products}), and as a consequence we obtain examples of groupoids which are not groups, do not have polynomial growth, and yet satisfy (RD).
Other than this, the main result is Theorem \ref{propreghoms}, which gives conditions on which (RD) transfers from the domain of a groupoid homomorphism to its codoman, and give a few corollaries to this result.  
But first, we give a simple result regarding inclusions of groupoids.

\begin{proposition}
    Suppose $\H$ is an \'etale groupoid, and that $\G\subset\H$ is an open subgroupoid.  
    Let $\ell$ be a length function on $\H$, let $\mathcal{F}\overset{\pi}{\to}\H$ be a twist over $\H$, and let $\E=\pi^{-1}(\G)$.  
    Then $\E$ is a twist over $\G$, and if $\H$ has property $\mathcal{F}$-(RD) with respect to $\ell$, then $\G$ has property $\E$-(RD) with respect to $\tilde{\ell}$, where $\tilde{\ell}$ is the restriction of $\ell$ to $\G$.
\end{proposition}
\begin{proof}
    It is straightforward to check that $\E$ defines a twist over $\G$.
    As $\G\subset\H$ is open, $\E$ is an open subgroupoid of $\mathcal{F}$, so extension-by-zero yields inclusions $\ScGE\subset\Sigma_c(\H,\mathcal{F})$ and $\ell^2\G_x\subset\ell^2\H_x$ for all $x\in\Go$.  Moreover these later inclusions are isometric.
    \\\indent Let $\rho:\H\to\mathcal{F}$ be a section for the bundle map $\pi$, and let $\tilde{\rho}$ denote the restriction of $\rho$ to $\G$.   
    For any $f\in\ScGE$, $x\in\Go$, and $\xi\in\ell^2\G_x$, we have
    \begin{align*}
        \|\lambda_x^{\tilde{\rho}}(f)\xi\|_{\ell^2\G_x}^2
        &=\|\lambda_x^\rho(f)\xi\|_{\ell^2\H_x}^2
        \le\|\lambda_x^\rho(f)\|_{\bbB(\ell^2\H_x)}\|\xi\|_{\ell^2\H_x}\\
        &=\|\lambda_x^\rho(f)\|_{\bbB(\ell^2\H_x)}\|\xi\|_{\ell^2\G_x}.
    \end{align*}
    It follows that $\|\lambda_x^{\tilde{\rho}}(f)\|_{\bbB(\ell^2\G_x)}\le\|\lambda_x^\rho(f)\|_{\bbB(\ell^2\H_x)}$, and thus $\|f\|_{C^*_r(\G,\E)}\le\|f\|_{C^*_r(\H,\mathcal{F})}$.  
    \\\indent For $t\ge0$ and $f\in\ScGE$, we have $\|f\|_{\mathcal{F},\ell,t,s,x}=\|f\|_{\E,\tilde{\ell},t,s,x}$ whenever $x\in\Go$  and $\|f\|_{\mathcal{F},\ell,t,s,x}=0$ whenever $x\in\Ho\setminus\Go$.  
    Taking suprema, it follows that $\|f\|_{\mathcal{F},\ell,t,s}=\|f\|_{\E,\tilde{\ell},t,s}$, and taking adjoints we obtain $\|f\|_{\mathcal{F},\ell,t}=\|f\|_{\E,\tilde{\ell},t}$.  
    Assuming $\H$ has $\mathcal{F}$-(RD) with respect to $\ell$, there are constants $C,t\ge0$ such that $\|h\|_{C^*_r(\H,\mathcal{F})}\le C\|h\|_{\mathcal{F},\ell,t}$ for all $h\in\Sigma_c(\H,\mathcal{F})$, and thus 
    \begin{align*}
    \|f\|_{\CrGE}\le\|f\|_{C^*_r(\H,\mathcal{F})}\le C\|f\|_{\mathcal{F},\ell,t}=C\|f\|_{\E,\tilde{\ell},t}.
    \end{align*}
\end{proof}

We now consider the case of products of \'etale groupoids.  
It is known (see Lemma 3.1 of \cite{CPSC2007} for instance) that products of (RD) groups satisfy (RD).  
At the time of this writing, it is not known to the author if the same holds in this more general setting, but we can obtain a partial result.

Before stating the result, let us note the following simple construction.  
Let $\G$ and $\H$ be \'etale groupoids, and let 
\begin{center}
    \begin{tikzcd}
        \Go\times\T \arrow[r, "i"] & \E \arrow[r, "\pi"] & \G
    \end{tikzcd}
\end{center}
be a twist over $\G$.  
One can define a twist over $\G\times\H$ by
\begin{center}
    \begin{tikzcd}
        \Go\times\Ho\times\T \arrow[r, "\tilde{i}"] & \E\times\H \arrow[r, "\tilde{\pi}"] & \G\times\H,
    \end{tikzcd}
\end{center}
where $\tilde{i}(x,y,z)=(i(x,z),z)$ for $x\in\Go$, $y\in\Ho$, $z\in\T$, and where $\tilde{\phi}(\ve,\eta)=(\pi(\ve),\eta)$ for $\ve\in\E$, $\eta\in\H$.

\begin{proposition}\label{products}
    Let $\G$ and $\H$ be \'etale groupoids, and let $\E$ be a twist over $\G$.   
    Suppose that $\H$ is compact, and that $\G$ has property $\E$-(RD) with respect to the length $\ell$.  
    Then $\G\times\H$ has property $\E\times\H$-(RD) with respect to the length function $\tilde\ell$, where $\tilde\ell(\gamma,\eta)=\ell(\gamma)$.
\end{proposition}
\begin{proof}
    Fix a finite cover $\{U_1,\ldots,U_n\}$ of $\H$ by open bisections, and let $(h_1,\ldots,h_n)$ be a partition of unity for $\H$ subordinate to $(U_1,\ldots,U_n)$. 
    For $f\in\Sigma_c(\G\times\H,\E\times\H)$, define $f^{(k)}\in\Sigma_c(\G\times\H,\E\times\H)$ for $k\in[n]:=\{1,\ldots,n\}$ by $f^{(k)}(\ve,\eta)=f(\ve,\eta)h_k(\eta)$, and define $f_\eta\in\ScGE$ for $\eta\in\H$ by $f_\eta(\ve)=f(\ve,\eta)$.  
    For $x\in\Go$, $y\in\Ho$, $\xi\in\ell^2(\G\times\H)_{(x,y)}$, and $\eta\in\H_y$, define $\xi_\eta\in\ell^2\G_x$ by $\xi_\eta(\gamma)=\xi(\gamma,\eta)$.  
    \\\indent Now fix $f\in\Sigma_c(\G\times\H,\E\times\H)$, and let $(x,y)\in\Go\times\Ho$ and $\xi\in\ell^2(\G\times\H)_{(x,y)}$ be given.  
    Let $Z=\{(\eta,k)\in\H_y\times[n]:r(\eta)\in r(U_k)\}$.  
    For $(\eta,k)\in Z$, there is a unique $\zeta=\zeta(\eta,k)\in\H_y$ such that $\eta\zeta^{-1}\in U_k$.  
    Let $\rho:\G\to\E$ be a section for the bundle map $\pi$, and let $\tilde{\rho}=\rho\times\id_{\H}$ be the corresponding section for the bundle map $\tilde{\pi}$.  Observe that for $\gamma\in\G_x$ and $(\eta,k)\in Z$ we have 
    \begin{align*}
        [\lambda_{(x,y)}^{\tilde{\rho}}(f^{(k)})\xi](\gamma,\eta)
        =h_k(\eta\zeta(\eta_2,k)^{-1})[\lambda_{x}^{\rho}(f_{\eta\zeta(\eta,k)^{-1}})\xi_{\zeta(\eta,k)}](\gamma).
    \end{align*}
    We estimate:
    \begin{align*}
        \|\lambda_{(x,y)}^{\tilde{\rho}}(f)\xi\|_{\ell^2(\G\times\H)_{(x,y)}}^2
        &=\sum_{\gamma\in\G_x}\sum_{\eta\in\H_y}\left|\sum_{k=1}^n[\lambda_{(x,y)}^{\tilde{\rho}}(f^{(k)})\xi](\gamma,\eta)\right|^2\\
        &\le n\sum_{(\eta,k)\in Z}\sum_{\gamma\in\G_x}|[\lambda_{(x,y)}^{\tilde{\rho}}(f^{(k)})\xi](\gamma,\eta)|^2\\
        &\le n\sum_{(\eta,k)\in Z}\|\lambda_{x}^{\rho}(f_{\eta\zeta(\eta,k)^{-1}})\xi_{\zeta(\eta,k)}\|_{\ell^2\G_x}^2\\
        &\le n\sum_{(\eta,k)\in Z}\|f_{\eta\zeta(\eta,k)^{-1}}\|_{C^*_r\G}^2\|\xi_{\zeta(\eta,k)}\|_{\ell^2\G_x}^2.
    \end{align*}
    Since $\G$ has property $\E$-(RD) with respect to $\ell$, there exist constants $C,t\ge0$ such that $\|g\|_{\CrGE}\le C\|g\|_{\E,\ell,t}$ for all $g\in \ScGE$.  
    For $\eta\in\H$, one checks that $\|f_{\eta}\|_{\E,\ell,t}\le\|f\|_{\E\times\H,\tilde{\ell},t}$, and hence
    \begin{align*}
        \|\lambda_{(x,y)}^{\tilde{\rho}}(f)\xi\|_{\ell^2(\G\times\H)_{(x,y)}}^2
        \le nC^2\|f\|_{\E\times\H,\tilde{\ell},t}^2\sum_{(\eta,k)\in Z}\|\xi_{\zeta(\eta,k)}\|_{\ell^2\G_x}^2.
    \end{align*}
    For fixed $k$, the map $\eta\mapsto\zeta(\eta,k)$ is injective, and thus 
    \begin{align*}
        \sum_{(\eta,k)\in Z}\|\xi_{\zeta(\eta,k)}\|_{\ell^2\G_x}^2\le n\|\xi\|_{\ell^2(\G\times\H)_{(x,y)}}^2.
    \end{align*}
    Combining our estimates, we have obtain 
    \begin{align*}
        \|\lambda_{(x,y)}^{\tilde{\rho}}(f)\xi\|_{\ell^2(\G\times\H)_{(x,y)}}^2
        \le nC\|f\|_{\E\times\H,\tilde{\ell},t}\|\xi\|_{\ell^2(\G\times\H)_{(x,y)}}.
    \end{align*}
    Taking the supremum over $\xi\in\ell^2(\G\times\H)_{(x,y)}$, then over $(x,y)\in\Go\times\Ho$, we obtain 
    \begin{align*}
        \|f\|_{C^*_r(\G\times\H,\E\times\H}\le nC\|f\|_{\E\times\H,\tilde{\ell},t}.
    \end{align*}
\end{proof}

This allows us to conclude that when $\G=\Gamma$ is a (discrete) group with property (RD), $\E$ is the trivial twist, and $\H$ is any compact \'etale groupoid, we see that $\Gamma\times\H$ admits property (RD). 
In particular, if $\F$ is a finitely generated free group, given the word length function with respect to a free generating set, and $\mathcal R_n$ denotes the full equivalence relation on the $n$-set $[n]=\{1,\cdots,n\}$ ,then $\F\times\mathcal R_n$ is a property (RD) groupoid which fails to have polynomial growth.  
This to be expected, as $C^*_r(\F\times\mathcal{R}_n)\cong M_n(C^*_r\F)$, and  in the language of Theorem 4.2 in \cite{Hou2017}, $S^{\ell}_2(\F\times\mathcal{R}_n)$ is just $M_n(S^{\ell}_2(\F_2))$, and spectral invariance of this subalgebra is known by Theorem 2.1 in \cite{Sch1992}.

Next, we consider a fairly simple observation regarding the relationship between property (RD) for translation groupoids and property (RD) for the acting group.

\begin{proposition}\label{grpacts}
    Let $\Gamma$ be a discrete group with a length function $\ell$, and suppose $\Gamma$ acts on a compact space $X$.  
    If $\Gamma\ltimes X$ has (RD) with respect to the length function induced by $\ell$, then $\Gamma$ has (RD) with respect to $\ell$.
\end{proposition}
\begin{proof}
    Let us write $\G=\Gamma\ltimes X$, and define the length function $\ell_{\Gamma\curvearrowright X}$ on $\Gamma\ltimes X$ by $\ell_{\Gamma\curvearrowright X}(\gamma,x)=\ell(\gamma)$ for all $\gamma\in\Gamma$ and $x\in X$.  
    If $f\in\C\Gamma$, define $\hat f\in\CcG$ by $\hat f(\gamma,x)=f(\gamma)$.  
    For any $x\in X$ and $\xi\in\ell^2\Gamma$, define $\hat\xi_x\in\ell^2\G_x$ by $\hat\xi_x(\gamma,x)=\xi(\gamma)$ for all $\gamma\in\Gamma$.  
    Moreover, one checks that 
    \begin{align*}
        [\lambda_x(\hat f)\hat\xi_x](\gamma,x)=[\lambda(f)\xi](\gamma), \qquad \|\hat f\|_{\ell_\G,t,s,x}=\|\hat f\|_{\ell_\G,t,s,x}=\|f\|_{\ell,t}
    \end{align*}
    for all $f\in\C\Gamma$, $\xi\in\ell^2\Gamma$, $\gamma\in\Gamma$, $x\in X$, and $t\ge0$.  
    If now $\G$ has (RD) with respect to $\ell_{\Gamma\curvearrowright X}$, there exist $C,t\ge0$ such that $\|g\|_{C^*_r\G}\le C\|g\|_{\ell_{\Gamma\curvearrowright X},t}$ for all $g\in\CcG$.   
    If now $f\in\C\Gamma$, we have 
    \begin{align*}
        \|f\|_{C^*_r\Gamma}=\|\hat f\|_{C^*_r\G}\le C\|\hat f\|_{\ell_{\Gamma\curvearrowright X},t}=C\|f\|_{\ell,t}.
    \end{align*}
    Since $C,t\ge0$ did not depend on $f$, the result follows.
\\\end{proof}

By Theorems \ref{princRD} and \ref{topprincRD}, the converse to this result fails drastically.  
For the time being, we shall attempt to generalize this result as much as possible.  
To that end, let $\H,\G$ be groupoids, and let $\phi:\H\to\G$ be a homomorphism.
Given a twist $\E$ over $\G$, we can construct a pullback twist $\phi^*\E$ over $\H$, making the following diagram commute:
\begin{center}
    \begin{tikzcd}
        \Ho\times\T \arrow[d] \arrow[r, "\phi^*i"] & \phi^*\E \arrow[d, "\phi^*"] \arrow[r, "\phi^*\pi"] & \H \arrow[d, "\phi"] \\
        \Go\times\T \arrow[r, "i"] & \E \arrow[r, "\pi"] & \G.                     
\end{tikzcd}
\end{center}
Here, $\phi^*\E=\E\tensor[_\pi]{*}{_\phi}\H$ is the fibered product, and all induced maps are the obvious ones.  
Note that the twist $\E\times\H$ over $\G\times\H$ considered in Proposition \ref{products} is just the pullback of the twist $\E$ along the projection $\G\times\H\to\G$ onto the first factor.  

We say that $\phi$ is \textit{$n$-regular} (for some $n\in\N$) if $\phi(\Ho)=\Go$, and $|\phi^{-1}(\gamma)\cap\H_y|=n$ for all $y\in\Ho$ and $\gamma\in\G_{\phi(y)}$.  
Note that this also implies $\phi$ is surjective, and that $|\phi^{-1}(\gamma)\cap\H^y|=n$ for all $y\in\Ho$ and $\gamma\in\G^{\phi(y)}$.

It is worth noting that 1-regular groupoid homomorphisms correspond to groupoid actions.  
Indeed, if $\G$ acts on a set $Y$, then the projection map $\pi:\G\ltimes Y\to\G$ is 1-regular.  
Conversely, let $\phi:\H\to\G$ be a 1-regular groupoid homomorphism, and set $p=\phi^{(0)}$ and $Y=\Ho$.  
If $y\in Y$ and $\gamma\in\G_{p(y)}$, then there is a unique $\eta_{\gamma,y}\in\H_y$ such that $\phi(\eta_{\gamma,y})=\gamma$, and the action of $\G$ on $Y$ is given by $\gamma\cdot y=r(\eta_{\gamma,y})$. 
Moreover, the map $\H\to\G\ltimes Y$, $\eta\mapsto(\phi(\eta),s(\eta))$, is a groupoid isomorphism.

\begin{lemma}
    Let $\G$ and $\H$ be \'etale groupoids, let $\E$ be a twist over $\G$ with section $\rho:\G\to\E$, and let $\phi:\H\to\G$ be an $n$-regular groupoid homomorphism.
    \begin{enumerate}[label = (\roman*)]
        \item If $y\in\Ho$ and $\xi\in\C\G_{\phi(y)}$, then $\hat\phi_y\xi:=\xi\circ\phi$ belongs to $\C\H_y$.  Moreover, the mapping $\xi\mapsto\hat\phi_y\xi$ extends to a bounded linear map $\ell^2\G_{\phi(y)}\to\ell^2\H_y$ such that $\|\hat\phi_y\xi\|_{\ell^2\H_y}=n^{1/2}\|\xi\|_{\ell^2\G_{\phi(y)}}$ for all $\xi\in\ell^2\G_{\phi(y)}$.
        \item If $\phi$ is continuous and proper, then $\|\hat\phi f\|_{\phi^*\E,\phi^*\ell,t}=n^{1/2}\|f\|_{\E,\ell,t}$ for any $t\ge0$ and $f\in\ScGE$ and any length $\ell$ on $\G$.
        \item If $\phi$ is continuous and proper, then  $\lambda_y^{\phi^*\rho}(\hat\phi f)\hat\phi_y=n\hat\phi_y\lambda_{\phi(y)}^\rho(f)$ for any $f\in\ScGE$ and $y\in\Ho$.
    \end{enumerate}
\end{lemma}
\begin{proof}
    If $y\in\Ho$ and $\xi\in\C\G_{\phi(y)}$, then 
    \begin{align*}
        |\supp(\hat\phi_y\xi)|=|\phi^{-1}(\supp(\xi))\cap\H_y|=n|\supp(\xi)|,
    \end{align*}
    so $\hat\phi_y\xi\in\C\H_y$.  Moreover, we have 
     \begin{align*}
        \|\hat\phi_y\xi\|_{\ell^2\H_y}^2
        =\sum_{\eta\in\H_y}|\hat\phi_y\xi(\eta)|^2
        =\sum_{\gamma\in\G_{\phi(y)}}|\phi^{-1}(\gamma)\cap\H_y|\cdot|\xi(\gamma)|^2
        =n\|\xi\|_{\ell^2\G_{\phi(y)}}^2,
    \end{align*}
    and {\em (i)} follows in the familiar way.  Now assume that $\phi$ is a continuous, proper, $n$-regular groupoid homomorphism.  
    Fix a section $\rho:\G\to\E$ for the bundle map $\E\to\G$.
    If $f\in\ScGE$ and $y\in\Ho$, then proceeding as in the above calculation, one sees that
    \begin{align*}
        \|\hat\phi f\|_{\phi^*\E,\phi^*\ell,t,s,y}=n^{1/2}\|f\|_{\E,\ell,t,s,\phi(y)}
    \end{align*}
    for all $y\in\Ho$.  
    Since $\phi$ maps $\Ho$ onto $\Go$, we obtain
    \begin{align*}
        \|\hat\phi f\|_{\phi^*\E,\phi^*\ell,t,s}
        &=\sup_{y\in\Ho}\|\hat\phi f\|_{\phi^*\E,\phi^*\ell,t,s,y}
        =n^{1/2}\sup_{y\in\Ho}\|f\|_{\E,\ell,t,s,\phi(y)}\\
        &=n^{1/2}\sup_{x\in\Go}\|f\|_{\E,\ell,t,s,x}
        =n^{1/2}\|f\|_{\E,\ell,t,s}.
    \end{align*}
     As $\hat\phi$ is a $*$-homomorphism, we have 
     \begin{align*}
        \|(\hat\phi f)^*\|_{\phi^*\E,\phi^*\ell,t,s}
        =\|\hat\phi(f^*)\|_{\phi^*\E,\phi^*\ell,t,s}
        =n^{1/2}\|f^*\|_{\E,\ell,t,s},
    \end{align*}
    and {\em (ii)} follows.  
    To prove {\em (iii)}, let $f\in\ScGE$, $y\in\Ho$, $\xi\in\ell^2\G_{\phi(y)}$, and $\eta\in\H_y$ be given.  
    We have 
    \begin{align*}
        \left[\lambda_{y}^{\phi^*\rho}(\hat\phi f)(\hat\phi_y\xi)\right](\eta)
        &=\sum_{\kappa\in\H_y}f\left(\rho(\phi(\eta))\rho(\phi(\kappa))^{-1}\right)\xi(\phi(\kappa))\\
        &=\sum_{\gamma\in\G_{\phi(y)}}|\phi^{-1}(\gamma)\cap\H_y|\cdot f(\rho(\phi(\eta))\rho(\gamma)^{-1})\xi(\gamma)\\
        &=n\sum_{\gamma\in\G_{\phi(y)}}f(\rho(\phi(\eta))\rho(\gamma)^{-1})\xi(\gamma)\\
        &=n[\lambda_{\phi(y)}^\rho(f)\xi](\phi(\eta))\\
        &=n[\hat\phi_y\lambda_{p(y)}^\rho(f)\xi](\eta).
    \end{align*}
\end{proof}

\begin{theorem}\label{propreghoms}
    Let $\phi:\H\to\G$ be a homomorphism of \'etale groupoids that is continuous, proper, and $n$-regular for some $n\in\N$.  
    Let $\E$ be a twist over $\G$.
    If $\ell$ is a length function on $\G$, and $\H$ has property $\phi^*\E$-(RD) with respect to the length $\phi^*\ell$, then $\G$ has property $\E$- (RD) with respect to $\ell$. 
\end{theorem}
\begin{proof}
    Let $f\in\ScGE$, $x\in\Go$, and $\xi\in\ell^2\G_x$ with $\|\xi\|_{\ell^2\G_x}=1$ be given.  
    Let $\rho:\G\to\E$ be a section for the bundle map $\E\to\G$.
    Fix some $y\in\Ho$ such that $\phi(y)=x$. 
    Applying {\em (i)} and {\em (iii)} of the previous lemma, we see that
    \begin{align*}
        \|\lambda_x^\rho(f)\xi\|_{\ell^2\G_x}
        &=n^{-1/2}\|\hat\phi_y\lambda_x^\rho(f)\xi\|_{\ell^2\H_y}\\
        &=n^{-3/2}\|\lambda_y^{\phi^*\rho}(\hat\phi f)\hat\phi_y\xi\|_{\ell^2\H_y}\\
        &\le n^{-3/2}\|\lambda_y^{\phi^*\rho}(\hat\phi f)\|_{\bbB(\ell^2(\H_y)}\|\hat\phi_y\xi\|_{\ell^2\H_y}\\
        &\le n^{-1}\|\hat\phi f\|_{C^*_r(\H,\phi^*\E)}
    \end{align*}
    Taking suprema, we obtain $\|f\|_{C^*_r(\G,\E)}\le n^{-1}\|\hat\phi f\|_{C^*_r(\H,\phi^*\E)}$.  
    Assuming $\H$ has property $\phi^*\E$-(RD) with respect to $\phi^*\ell$, there exist constants $C,t\ge0$ such that $\|h\|_{C^*_r(\H,\phi^*\E)}\le C\|h\|_{\phi^*\E,\phi^*\ell,t}$ for all $h\in\Sigma_c(\H,\phi^*\E)$.  
    Applying {\em (ii)} from the previous lemma, the above estimate now yields
    \begin{align*}
        \|f\|_{\CrGE}
        \le n^{-1}\|\hat\phi f\|_{C^*_r(\H,\phi^*\E}
        \le Cn^{-1}\|\hat\phi f\|_{\phi^*\ell,t}
        =Cn^{-1/2}\|f\|_{\ell,t}
    \end{align*}
    As $f\in\ScGE$ was arbitrary, $\G$ has (RD) with respect to $\ell$.
\\\end{proof}

First, observe that Proposition \ref{grpacts} is a corollary of the above result:  
The quotient map $\pi:\Gamma\ltimes X\to\Gamma$ is $1$-regular for every continuous action, and proper when the space $X$ is compact.

Generalizing the above corollary, we now turn our attention to groupoid actions.  
Let $\G$ be an \'etale groupoid, and suppose it admits a left action on the locally compact Hausdorff space $Y$ with anchor map $p:Y\to\Go$.  
Let $\H=\G\ltimes Y$, and let $\pi:\H\to\G$ be the projection map: $\pi(\gamma,y)=\gamma$.  
Then $\pi$ is a continuous and 1-regular groupoid homomorphism, so for the above result to apply we only need to supply conditions for $\pi$ to be a proper map.  
This turns out to be the case when $p$ is a finite cover.  
To prove this, we require a lemma.

\begin{lemma}\label{dirsets}
    Let $I$ be a directed set, and let $I_1,\ldots,I_n$ be subsets of $I$.  
    If the union $\cup_{k=1}^nI_k$ is cofinal in $I$, then there is some $k\in\{1,\ldots,n\}$ such that $I_k$ is cofinal in $I$.  
\end{lemma}
\begin{proof}
    Write $I_0=I_1\cup\cdots\cup I_n$.  
    The result is trivial if $n=1$, so assume $n\ge 2$, and suppose that $I_1,\ldots,I_{n-1}$ are not cofinal in $I$.  
    Then for each $k\in\{1,\ldots,n-1\}$, there is some $i_k\in I$ such that whenever $i\in I$ and $i\ge i_k$ we must have $i\notin I_k$.  
    Fix some $i_0\in I$ such that $i_0\ge i_k$ for each $k<n$.  
    Let $i\in I$ be given, and let $i_n\in I$ be such that $i_n\ge i$ and $i_n\ge i_0$. 
    For each $k\in\{1,\ldots,n-1\}$, we have $i_n\ge i_k$, so $i_n\notin I_k$.  
    This forces $i_n\in I_n$, and it follows that $I_n$ is cofinal in $I$.  
\\\end{proof}

\begin{proposition}\label{propprojaction}
    Let $\G$ be an \'etale groupoid, and suppose $\G$ acts on a locally compact space $X$ such that the anchor map $p:X\to\Go$ is a finite cover.  
    Then the projection map $\pi:\G\ltimes X\to\G$ is proper. 
\end{proposition}
\begin{proof}
    Let $K\subset\G$ be compact, and let $(\gamma_i)_{i\in I}$ be a net in $\pi^{-1}(K)$.  
    For each $i\in I$ write $\gamma_i=(\sigma_i,x_i)$, where $\sigma_i\in K$ and $x_i\in X$.  
    As $K$ is compact, by passing to a subnet we may assume that the net $(\sigma_i)_{i\in I}$ converges to some $\sigma\in K$.  
    \\\indent Let us write $p^{-1}(s(\sigma))=\{x_1,\ldots,x_n\}$, and fix an open neighborhood $V$ of $s(\sigma)$ in $\Go$ that is evenly covered by the sets $\{V_1,\ldots,V_n\}$, where for each $k\in\{1,\ldots,n\}$, $V_k\subset X$ is an open neighborhood of $x_k$.  
    For $1\le k\le n$ let us write $I_k=\{i\in I:x_i\in V_k\}$, and $I_0=I_1\sqcup\cdots\sqcup I_n$.  
    As $I_0=\{i\in I:s(\gamma_i)\in V\}$, and $s(\gamma_i)\to s(\gamma)$, $I_0$ is cofinal in $I$.  
    Lemma \ref{dirsets} now implies that there is some $k\in\{1,\ldots,n\}$ such that $I_k$ is cofinal in $I$.  
    \\\indent Let $U\subset X$ be an open neighborhood of $x_k$.  
    Then $p(U\cap V_k)$ is an open neighborhood of $s(\sigma)$, so there is some $i_0\in I$ such that $p(x_i)\in p(U\cap V_k)$ whenever $i\ge i_0$.  
    Since $I_k$ is cofinal in $I$, we may assume that $i_0\in I_k$.  
    Thus $x_i\in U\cap V_k$ whenever $i\in I_k$ and $i\ge i_0$.   
    It follows that the net $(\gamma_i)_{i\in I_k}$ converges to $(\sigma,x_k)\in\pi^{-1}(K)$, and therefore $\pi^{-1}(K)$ is compact.  
\\\end{proof}

\begin{corollary}
    Let $\G$ be an \'etale groupoid, and let $\ell$ be a length function on $\G$.  
    Suppose that $\G$ admits an action on a locally compact space $Y$ and that the anchor map $p:Y\to\Go$ is a finite covering map.  
    If $\G\ltimes Y$ has property (RD) with respect to the length function induced by $\pi^*\ell$, then $\G$ has property (RD) with respect to $\ell$.
\end{corollary}

As a last application, we consider blow ups.  
Let $\G$ be an \'etale groupoid, let $Y$ be a locally compact space, and let $p:Y\to\Go$ be a surjective local homeomorphism.  
We denote by $\G[p]$ the blow up of $\G$ by the map $p$.  
This is, by definition, the groupoid whose underlying space is $Y\tensor[_p]{*}{_r}\G\tensor[_s]{*}{_p}Y$, with the obvious groupoid operations.  
With the subspace topology coming from the product topology on $Y\times\G\times Y$, $\G[p]$ is an \'etale groupoid with unit space homeomorphic to $Y$.

Given a local homeomorphism $p:Y\to\Go$, we define a map $p_0:\G[p]\to\G$ by $p_0(w,\gamma,y)=\gamma$.  
This is a continuous groupoid homomorphism.  
Moreover, we have a result similar to the above lemma for group actions and finite covers.

\begin{proposition}
    Let $\G$ be an \'etale groupoid, and let $p:Y\to\Go$ be an $n$-fold covering map.  
    Then the map $p_0:\G[p]\to\G$ is an $n$-regular and proper groupoid homomorphism.
\end{proposition}
\begin{proof}
    It is clear that $p_0$ is $n$-regular when $p$ is an $n$-fold cover.  
    The proof that $p_0$ is a proper map is similar to the proof of Proposition \ref{propprojaction}, and will be omitted.
\\\end{proof}

\begin{corollary}
    Let $\G$ be an \'etale groupoid with a length function $\ell$, and let $p:Y\to\Go$ be a finite covering map.  If $\G[p]$ has property (RD) with respect to the length induced by $\ell$, then $\G$ has (RD) with respect to $\ell$.  
\end{corollary}

\end{document}